\documentclass[12pt]{article}

\usepackage{amsmath}
\usepackage{amsfonts}
\usepackage{latexsym}
\usepackage{amssymb}
\usepackage{enumerate}
\usepackage{amsthm}
\usepackage{csquotes}
\usepackage{lineno,xcolor}
\usepackage[normalem]{ulem}
\usepackage{filecontents}
\usepackage{scalefnt}
\usepackage{tikz}
\usepackage{fancyhdr}
\usetikzlibrary{arrows}
\usetikzlibrary{positioning}

\usepackage{lineno,xcolor}
\definecolor{amaranth}{rgb}{0.9, 0.17, 0.31}
\definecolor{bluegray}{rgb}{0.4, 0.6, 0.8}

\makeatletter

\newtheorem*{maintheorem*}{Main Theorem}
\newtheorem{theorem}{Theorem}[section]
\newtheorem{proposition}[theorem]{Proposition}
\newtheorem{corollary}[theorem]{Corollary}
\newtheorem{lemma}[theorem]{Lemma}

\newtheorem*{theorem*}{Theorem}

\newtheorem{remark}[theorem]{Remark}

\newtheorem*{example*}{Example}
\newtheorem*{conjecture*}{Conjecture}
\def\1{\mathbf 1}

\def\0{\mathbf 0}

\def\cB{\mathcal B}
\def\cC{\mathcal C}
\def\cD{\mathcal D}

\def\cG{\mathcal G}

\def\cL{\mathcal L}

\def\cO{\mathcal O}
\def\cP{\mathcal P}

\def\cS{\mathcal S}

\def\bP{{\mathbb P}}

\def\PG{{\rm PG}}

\def\PGammaL{{\rm P\Gamma L}}

\def\GF{{\rm GF}}

\def\mM{{\rm M}}

\def\Aut{{\rm Aut}}

\def\GF{{\rm GF}}

\def\<{\langle}
\def\>{\rangle}
\newcommand\comment[1]{}

\newcommand*{\shifttext}[2]{
  \settowidth{\@tempdima}{#2}
  \makebox[\@tempdima]{\hspace*{#1}#2}
}

\newcommand\redsout{\bgroup\markoverwith{\textcolor{amaranth}{\rule[0.5ex]{2pt}{0.4pt}}}\ULon}

\newcommand\redout{\bgroup\markoverwith
{\textcolor{red}{\rule[.4ex]{2pt}{0.8pt}}}\ULon}

 \makeatletter
\def\@fnsymbol#1{\ensuremath{\ifcase#1\or *\or \dagger\or \ddagger\or
   \mathsection\or \mathparagraph\or \|\or **\or \dagger\dagger
   \or \ddagger\ddagger \else\@ctrerr\fi}}
    \makeatother

\title{Classifying pseudo-ovals, translation generalized quadrangles, and elation Laguerre planes \\of small order}
\author{Giusy Monzillo\thanks{The research is supported in part by the Ministry of Education, Science and Sport of Republic of Slovenia (University of Primorska Developmental funding pillar).} \\
\small {\tt giusy.monzillo@famnit.upr.si}\\[0.8ex]
\small UP FAMNIT\\[-0.8ex]
\small University of Primorska\\[-0.8ex]
\small Glagolj\v aska 8 \\[-0.8ex]
\small 6000 Koper, Slovenia
\and   
Tim Penttila \\ 
\small {\tt penttila86@msn.com}\\
\small School of Mathematical Sciences\\[-0.8ex]
\small The University of Adelaide\\[-0.8ex]
\small Adelaide, South Australia \\[-0.8ex]
\small 5005 Australia
\and
Alessandro Siciliano\footnote{The research is supported by the Italian National Group for Algebraic and Geometric Structures and their Applications (GNSAGA-INdAM).} \\
\small{\tt alessandro.siciliano@unibas.it}\\[0.8ex]
\small Dipartimento di Matematica, Informatica ed Economia\\[-0.8ex]
\small Universit\`a degli Studi della Basilicata\\[-0.8ex]
\small Viale dell'Ateneo Lucano 10 \\[-0.8ex]
\small 85100 Potenza, Italy\\
}
\date{}

\begin{document}

\maketitle

\thispagestyle{fancy}
\fancyhf{}
\renewcommand{\headrulewidth}{0pt}
\rhead{{\em{version 3.0}}}
\lhead{}

\begin{abstract}
We provide classification results for translation generalized quadrangles of order less or equal $64$, and hence, for all incidence geometries related to them. The results consist of the classification of all pseudo-ovals in $\PG(3n-1,2)$, for $n=3,4$, and that of the pseudo-ovals in $\PG(3n-1,q)$, for $n=5,6$, such that one of the associated projective planes is Desarguesian.
\end{abstract}

{\it Keywords: oval, pseudo-oval, generalized quadrangle, Laguerre plane}             

{\it AMS Class.: 51E21, 51E20, 51E12}

\section{Introduction and known results}\label{sec_1}

This paper can be considered as a sequel to \cite{mps1,mps2}, where the recent classification of hyperovals of $\PG(2,64)$ \cite{vander} was used to extend previously known classification results for related structures to order 64. 
Here, we classify all pseudo-ovals of $\PG(3n-1,2)$, $n=3,4$, and all pseudo-ovals of $\PG(3n-1,2)$, $n=5,6$, yielding  a Desarguesian translation plane via the  construction given in \cite[pp. 181-182]{pt3}.  As a by-product, thanks to the correspondence between pseudo-ovals and translation generalized quadrangles \cite[8.7.1]{pt3}, we also classify all translation generalized quadrangles of order $2^n$, $n=3,4$, and all translation generalized quadrangles $T(\cO)$ of order $q^n$, $n=5,6$, when the pseudo-oval $\cO$ gives rise to a Desarguesian translation plane. In the case $n=3$, our classification of translation generalized quadrangles is computer-free. Another consequence is the classification of all elation Laguerre planes of order $2^n$, $n=3,4$, and that of all elation Laguerre planes of order $2^n$, $n=5,6$, with a Desarguesian projective completion. 
We point out that the classification of elation Laguerre planes of order 16 has already been obtained by Steinke in \cite{ste4}, but here we present a different method of proof.

Since the results in this paper involve a certain number of incidence structures that are closely related to each other, it is worth recalling their definition and main properties together with the corresponding known classification results when the order of the structure is small. This is done in this section, and, to make everything easier and more readable, we divide it into subsections, one for each object we consider. All classification results as well as the tools and methods used to get them are instead collected in Section 2.

\subsection{Generalized quadrangles}

A (finite) {\em generalized quadrangle} (GQ) is an incidence structure $\cS=(\cP,\cB, {\rm I})$ where  $\cP$ and $\cB$ are disjoint (nonempty and finite) sets of objects called {\em points} and {\em lines}, respectively, and  $\rm I$ is a symmetric point-line relation, called {\em incidence relation}, satisfying the following axioms:
\begin{itemize} 
\item[ (i)] Each point is incident with $t+1$ lines, and two distinct points are incident with at most one line.
\item[(ii)]  Each line is incident with $s+1$ points, and two distinct lines are incident with at most one point.
\item[(iii)] If $x$ is a point and $L$ is a line not incident with $x$, then there is a unique pair $(y, M)\in \cP\times \cB$ such that $x\, {\rm I}\, M\, {\rm I}\, y\, {\rm I}\,  L$.
\end{itemize}

The integers $s$ and $t$ are the {\em parameters} of the GQ, and $\cS$ is said to have {\em order} $(s,t)$; if $s=t$, it is said to have order $s$.  If $\cS$ has order $(s,t)$, then it follows that $|\cP| = (s + 1)(st + 1)$ and $|\cB| = (t + 1)(st + 1)$. If $\cS=(\cP,\cB, {\rm I})$ is a GQ of order $(s,t)$, then the incidence structure $\cS^*=(\cB,\cP, {\rm I})$ is a GQ of order $(t,s)$, and it is called the {\em dual of}  $\cS$. 

The \emph{classical} GQs of order $q$ are $W(q)$, whose points and lines are those of a symplectic geometry in $\PG(3,q)$, and $Q(4,q)$, whose points and lines are those of a non-singular quadric in $\PG(4,q)$. It is known that $Q(4,q)$ is isomorphic to the dual of $W(q)$, and it is self-dual when $q$ is even \cite{pt3}. For more details on GQs, the reader is referred to \cite{pt3}.

There is a unique GQ of order 2 \cite[Chapter 6]{pt3}; there are exactly two GQs of order 3 \cite[Chapter 6]{pt3}; there is a unique GQ of order 4 \cite{p1,p2}, \cite[Chapter 6]{pt3}. Up to the knowledge of the authors, for no other value of $s$, there is a complete classification of GQs of order $s$.

A \emph{translation generalized quadrangle} (TGQ) with \emph{base point} $P$ is a GQ for which there is an abelian group acting regularly on the points not collinear with $P$, while fixing each line through $P$. Such a group is uniquely determined, and it is called the \emph{translation group at} $P$. The classical GQ $Q(4, q)$ is a TGQ, but, for $q$ odd, $W(q)$ is not a TGQ.  In \cite[8.7.3(i)]{pt3}, it is shown that a TGQ of prime order $p$ is isomorphic to $Q(4, p)$.

\comment{The base point of a TGQ of order $s$ is co-regular \cite[Section 8.6]{pt3}. Thus, the perp-geometry at any line on $P$ is a projective plane of order $s$ \cite[1.3.1]{pt3}. Indeed, this is a translation plane \cite[Chapter 8]{pt3}. If $s$ is even, the base point is regular \cite[Section 8.6]{pt3}, and the perp-geometry at $P$ is a projective plane of order $s$ \cite[1.3.1]{pt3}, and again, this is a translation plane \cite[Chapter 8]{pt3}.}

Two lines of a GQ are said to be {\em concurrent} if there is a point incident with both of them. A {\em symmetry} of a GQ with respect to the line $l$ is a collineation fixing each line concurrent with $l$. The {\em kernel} of a TGQ with base point $P$ is the ring of all endomorphisms of the translation group $T$ fixing the subgroup $T_l$ of all symmetries of $T$ with respect to $l$, for each line $l$ through $P$. With the usual addition and multiplication of endomorphisms, the kernel is a ring, thus by Wedderburn's Theorem  is a finite field. If the kernel is $\GF(q)$ and the TGQ has order $s$, then $s$ is a power of $q$. 
 
In \cite{pt1,pt2}, using a result of Denniston \cite{den} on ovals in the Hall plane of order 9, the
uniqueness of a TGQ of order 9 is proved. This depends on the classification of translation planes of order 9 \cite{bru, lun}.
Note that, by \cite[Section 8.6]{pt3}, the base point of a TGQ of odd order $s$ is anti-regular, so that a result by Steinke on Laguerre planes  \cite{ste2} implies the uniqueness of a TGQ  of order 9. In Corollary \ref{cor_2} we show that there are exactly two TGQs of order 8. Here we present a computer-free proof of this result. Furthermore,  Corollary \ref{cor_3}  shows that there are exactly three TGQs of order 16. In this case, the result is computer-based. About the cases $q=32, 64$, we refer the reader to Corollaries \ref{n_5} and \ref{n_6}.

\subsection{Pseudo-ovals}

An $n$-\emph{dimensional pseudo-oval} in $\PG(3n-1, q)$ is a set $\cO$ of $q^n + 1$ $(n-1)$-dimensional subspaces with the property that any three of them span $\PG(3n-1, q)$.
Given an $n$-dimensional pseudo-oval, there exists a unique system of  $q^n + 1$ $(2n-1)$-dimensional subspaces, called the \emph{tangents} to $\cO$, such that each element $X$ of $\cO$ is on a unique tangent $T(X)$, and $T(X)$ meets no other element of $\cO$ \cite{thas}. The {\em order} of an $n$-dimensional pseudo-oval in $\PG(3n-1, q)$ is $q^n$. When $n=1$, the definition of pseudo-oval covers that of an oval in $\PG(2,q)$. 

An oval in $PG(2, q^n)$ gives rise to an $n$-dimensional pseudo-oval in $\PG(3n-1, q)$, by identifying the underlying vector space $\GF(q)^{3n}$ of $\PG(3n-1, q)$ with the underlying vector space $\GF(q^n)^3$ (considered as a vector space over $\GF(q)$) of $\PG(2, q^n)$. A pseudo-oval so constructed is known as \emph{elementary pseudo-oval}. No other pseudo-ovals are known.  
It \cite{bms} it was proved that the only pseudo-ovals contained in the elliptic quadric $Q^-(5,7)$ are elementary, so arising from conics in $\PG(2,49)$ by Segre's Theorem \cite{se}.

\begin{proposition}\label{prop_1}
Let $\cO_1$ and $\cO_2$ be two elementary $n$-dimensional pseudo-ovals in $\PG(3n-1,q)$ arising from the ovals $\widehat\cO_1$ and $\widehat\cO_2$ in $\PG(2,q^n)$, respectively. Then, $\cO_1$ and $\cO_2$ are projectively equivalent if and only if $\widehat\cO_1$ and $\widehat\cO_2$ are.
\end{proposition}
\begin{proof}
Assume that there is a collineation $g$ of $\PG(3n-1,q)$ that maps $\cO_1$ to  $\cO_2$. Then, the kernel $K_1$ of $\cO_1$ is conjugate to the kernel $K_2$ of $\cO_2$; we may assume that both of them are the same copy $K$ of the multiplicative group of  $\GF(q^n)$.  Hence, $K$ is a Singer cycle of $\PGammaL(3n,q)$, and $g$ lies in the normalizer of $K$ in $\PGammaL(3n,q)$. From \cite[7.3 Seitz]{hu} $g$ lies in $\PGammaL(3,q^n)$, which is naturally embedded in $\PGammaL(3n,q)$ by identifying the underlying vector space $\GF(q^n)^{3}$ of $\PG(2,q^n)$ with the underlying vector space $\GF(q)^{3n}$  of $\PG(3n-1,q)$. This gives that the two ovals $\widehat\cO_1$ and $\widehat\cO_2$ are projectively equivalent. For the converse just note that any collineation of $\PG(2,q^n)$ naturally embeds in $\PGammaL(3n,q)$ via the identification of $\GF(q)^{3n}$ with $\GF(q^n)^{3}$.
\end{proof}

An $n$-dimensional pseudo-oval $\cO$ in $\PG(3n-1,q)$ gives rise to a number of $(n-1)$-spreads of $\PG(2n-1, q)$ \cite[pp. 181-182]{pt3}, and so to many translation planes of order $q^n$ \cite{an,bb}. For instance, the projections of the elements of $\cO$ from a fixed $X\in \cO$  onto a $(2n-1)$-dimensional subspace $S$ of $\PG(3n-1,q)$ skew to $X$, together with $T(X)\cap S$, form an $(n-1)$-spread of $S$. The corresponding translation plane, denoted by $\pi(\cO,X)$, has order $q^n$ with kernel containing $\GF(q)$ \cite{an,bb}.

Let $\cO$ be a $n$-dimensional pseudo-oval in $\PG(3n-1,q)$. Embed $\PG(3n-1,q)$ as a hyperplane $H$ in $\PG(3n,q)$, and let $T(\cO)$ be the incidence structure with points: (i) the points of $\PG(3n,q)$ not on $H$, (ii) the $2n$-spaces of $\PG(3n,q)$ meeting $H$ in a tangent to $\cO$, (iii) a new point $(\infty)$; with lines: (a) the $n$-spaces of $\PG(3n, q)$, not in $H$, meeting $H$ in an element of $\cO$, (b) the elements of $\cO$; and taking as the incidence relation the natural incidence in $\PG(3n-1,q)$ augmented by making all lines of type (b) incident with $(\infty)$. Then, $T(\cO)$ is a TGQ of order $q^n$ with base point $(\infty)$ and kernel containing $\GF(q)$ \cite[Th. 8.7.1]{pt3}. Conversely,  every TGQ of order $q^n$ with $\GF(q)$ contained in its kernel is isomorphic to $T(\cO)$, for some $n$-dimensional pseudo-oval $\cO$ in $\PG(3n-1,q)$ \cite[Th. 8.7.1]{pt3}. The proof in  \cite{okp3} generalizes to show that $T(\cO_1)$ is isomorphic to $T(\cO_2)$ if and only if $\cO_1$ is projectively equivalent to $\cO_2$.
When $n=1$, we use $T_2(\cO)$ to refer to the GQ $T(\cO)$, where $T_2(\cO)$ is the Tits quadrangle constructed on the oval $\cO$ in $\PG(2,q)$ \cite{dem}. 

The {\em kernel} of an $n$-dimensional pseudo-oval $\cO$ in $\PG(3n-1,q)$ is the kernel of the TGQ $T(\cO)$: it is a subfield of $\GF(q^n)$ and an extension of $\GF(q)$. It is known \cite[8.6.5]{pt3} that the multiplicative group of the kernel of $T(\cO)$ is isomorphic to the group of all whorls about $(\infty)$ fixing a given point not collinear with $(\infty)$; here,  a {\em whorl about the point $(\infty)$} is a collineation of $T(\cO)$  fixing each line incident with $(\infty)$. This yields that the kernel of $T(\cO)$ is $\GF(q^n)$ if and only if $\cO$ is elementary, i.e., it arises from an oval $\cO'$ of $\PG(2,q^n)$, in which case $T(\cO)$ is isomorphic to $T_2(\cO')$.

In analogy to the case $n=1$, if $q$ is odd, the tangents to an $n$-dimensional pseudo-oval in $PG(3n-1, q)$ form a pseudo-oval in the dual space, and, if $q$ is even, the tangents to an $n$-dimensional pseudo-oval $\cO$ in $\PG(3n-1, q)$ all pass through an $(n-1)$-space $N$, the \emph{nucleus} of the pseudo-oval. In addition, for any $X\in \cO$, $(\cO\setminus \{X\})\cup \{N\}$ is an $n$-dimensional pseudo-oval in $\PG(3n-1, q)$ \cite[Theorem 6(v)] {thas2}, \cite{thas, thas1}, \cite[pp.181-182]{pt3}.

An $n${\em -dimensional dual pseudo-oval} in $\PG(3n-1,q)$ is a set $O$ of $q^n+1$ $(2n-1)$-dimensional subspaces with the property that any three of them intersect trivially, together with $q^n+1$ $(n-1)$-dimensional subspaces, called the {\em tangents to} $O$, with the property that each element $X$ of $O$ contains a unique tangent $T(X)$, and $T(X)$ meets no other element of $O$. 
Under a correlation of $\PG(3n-1,q)$, the image of an $n$-dimensional dual pseudo-oval in $\PG(3n-1,q)$ is an $n$-dimensional pseudo-oval in $\PG(3n-1,q)$. Moreover, over fields of odd order, a pseudo-oval is a self-dual concept. The situation is quite different for fields of characteristic 2, and this explains some of the subtleties that occur.

It follows from the uniqueness of the GQ of order 4 that there is a unique 2-dimensional pseudo-oval in $\PG(5,2)$, and from the uniqueness of the TGQ of order 9 that there is a unique 2-dimensional pseudo-oval in $\PG(5,3)$.
There are exactly two 3-dimensional pseudo-ovals in $\PG(8, 2)$ (Theorem \ref{th_4}). They are elementary, so arising from the conic  and the pointed-conic in $\PG(2,8)$. There are exactly three 4-dimensional pseudo-ovals in $\PG(11, 2)$ (Theorem \ref{th_5}). All of the 4-dimensional pseudo-ovals in $\PG(11,2)$ are elementary, and so they arise from ovals in $\PG(2,16)$. Hence, there are exactly three 2-dimensional pseudo-ovals in $\PG(5, 4)$ (Corollary \ref{cor_4}).  About the cases $n=5, 6$, we refer the reader to Theorems \ref{th_n5}, \ref{th_n6}.

\subsection{Laguerre planes}\label{LP}

A {\em finite Laguerre plane} of order $n$, $n\ge 2$, is a triple  $\cL = (\cP,\cC,\cG)$  consisting of a set $\cP$  of $n(n+1)$ {\em points}, a set $\cC$ of $n^3$ {\em circles} and a set $\cG$ of $n+1$ {\em generators} (or {\em parallel classes}), where circles and generators are both subsets of $\cP$, such that the following three axioms are satisfied:
\begin{itemize}
\item[(i)] $\cG$ partitions $\cP$ and each generator contains $n$ points.
\item[(ii)] Each circle intersects each generator in precisely one point.
\item[(iii)] Three points no two of which are on the same generator can be uniquely joined by a circle.
\end{itemize}

Given a finite Laguerre plane $\cL$ of order $n$ and a point $P$ of the plane, the {\em derived affine plane at}  $P$  is the incidence structure $\cL_P$  whose points are the points of $\cL$ not on the generator through $P$, and the lines are the circles of $\cL$ through $P$  together with the generators not passing through $P$ (with the natural incidence relation). It turns out that $\cL_P$ is an affine plane of order $n$; its projective completion $\bP\cL_P$  is called the {\em derived projective plane at} $P$.

Kleinewillingh\"ofer \cite{klein} classifies Laguerre planes in terms of their respective automorphism group, this classification being analogous to the Lenz-Barlotti classification of projective planes and the Hering classification of inversive planes.
In \cite{chka} (see also \cite{pt1,pt2}), Chen and Kaerlein show that if the derived affine plane of a Laguerre plane of finite order $n$, where $n$ is 2 or 4 or odd, is Desarguesian, then the Laguerre plane is Miquelian. This has as a corollary the uniqueness of the Laguerre planes of orders 2, 3, 4, 5, and 7, relying on the uniqueness of the projective planes of these orders. In \cite{ste2}, the uniqueness of the Laguerre plane of order 9 is proved. This depends on the classification by computer of all projective planes of order 9 \cite{lkt}, and on the classification by computer of all ovals in these planes \cite{den, ckk}. 
There are exactly two Laguerre planes of order 8, and the only references in the literature we can find are \cite[p.530]{gs}, \cite[p.108]{gs2}, \cite[p.163]{ste3}. 
Since there is no affine plane of order 6, there is obviously no Laguerre plane of order 6. By applying the Bruck-Ryser theorem \cite{br}, more generally there is no Laguerre plane of order $n$, where $n$ is congruent to 1 or 2 modulo 4 and the square-free part of $n$ is divisible by a prime congruent to 3 modulo 4. This rules out Laguerre planes of orders 6, 14, 21, 22,... Moreover, by the famous computer-based result of the nonexistence of an affine plane of order 10 \cite{lts}, there is no Laguerre plane of order 10. In fact, since the projective completion of the derived affine plane of a Laguerre plane at a point contains an oval, we need only the earlier computer-based result \cite{ltsm} that there is no projective plane of order 10 containing an oval. Up to the knowledge of the authors, for no other $n$, there is a complete classification of Laguerre planes of order $n$.

 
Two Laguerre planes $\cL_1$ and $\cL_2$ are said to be \emph{isomorphic} if there is a bijection from the point set of $\cL_1$ to that of $\cL_2$ such that circles are mapped onto circles. An \emph{automorphism} of a Laguerre plane $\cL$ is an isomorphism of $\cL$ on itself. All automorphisms of $\cL$ form a group with respect to composition, the automorphism group $\Aut(\cL)$ of $\cL$.

An {\em elation Laguerre plane} is a Laguerre plane which admits an automorphism group acting regularly on the circles, while fixing every generator; such a group is uniquely determined and abelian \cite[Corollary 2.9, Proposition 2.8]{jos}. It is called the {\em elation group} of the Laguerre plane. 
The derived affine plane $\cL_P$ at any point $P$ of an elation Laguerre plane $\cL$ are dual translation planes \cite[Theorem 3]{ste3} \cite[Lemma 2.3]{jos}, and hence a finite elation Laguerre plane has prime power order. 
An elation Laguerre plane of prime order is Miquelian, by applying the theorem of Andr\'e \cite{an} that a translation plane of prime order is Desarguesian, and then applying the result in \cite{chka}. 

Models of elation finite Laguerre planes can be obtained as follows. 
Let $\cO$ be an oval in $\PG(2, q)$ 
 and $K$ a cone in $\PG(3, q)$ with vertex $V$ projecting $\cO$. The incidence structure $\cL_2(\cO)$ with points the points of $K$ other than $V$, generators the generators of $K$, circles the intersections of the planes of $\PG(3,q)$ not on $V$ with $K$ is a Laguerre plane of order $q$. The group of all elations of $\PG(3,q)$ with center $V$ fixes every line on $V$, and it is regular on the planes of $\PG(3, q)$ not on $V$; hence, it fixes every generator of $\cL_2(\cO)$, while acting regularly on the circles of $\cL_2(\cO)$. This yields that $\cL_2(\cO)$ is an elation Laguerre plane. A finite Laguerre plane is said to be  {\em ovoidal} if it is isomorphic to $\cL_2(\cO)$, for some oval $\cO$ of $\PG(2, q)$.  Every known finite Laguerre plane is ovoidal. By \cite{maz}, $\cL_2(\cO_1)$ is isomorphic to $\cL_2(\cO_2)$ if and only if $\cO_1$ is projectively equivalent to $\cO_2$ (that is, they are in the same orbit under the collineation group $\PGammaL(3, q)$ of $\PG(2,q)$). When $\cO$ is a conic, the elation Laguerre plane $\cL_2(\cO)$ is characterized among all Laguerre planes by satisfying the configuration of Miquel \cite{vdw},\cite[pp. 245-246]{dela}, and hence is called {\em Miquelian}. General references on Laguerre planes are \cite{be,hahe,dela,ste}.

It is of interest to give the dual construction (in $\PG(3, q)$)  of $\cL_2(\cO)$, for comparison with both the previous construction of the Tits GQ $T_2(\cO)$ and its generalization involving pseudo-ovals. Let $O$ be a dual oval of $\PG(2, q)$ (i.e., a set of $q+1$ lines, no three of them concurrent). Embed $\PG(2, q)$ as a plane $\pi$ in $\PG(3, q)$, and let $L(O)$ be the incidence structure with points the planes of $\PG(3, q)$ meeting $\pi$ in a line of $O$, generators the lines of $O$, circles the points of $\PG(3, q)$ not on $\pi$, and with incidence relation the natural one. Then, $L(O)$ is a Laguerre plane of order $q$, isomorphic to $\cL_2(O^*)$, where $O^*$ is the image of $O$ under a correlation.

The {\em kernel} of an elation Laguerre plane with elation group $E$ is the ring of all endomorphisms of $E$ fixing the subgroup $E_P$, for all points $P$ of the plane. The kernel of an elation Laguerre plane is a skew-field \cite[Theorem 2.10]{jos}. If the elation Laguerre plane is finite, then the kernel is a finite field $\GF(q)$ and the order of the plane is a power of $q$.

Inspired by Casse, Thas and Wild \cite{ctw}, Steinke \cite{ste3} and L\"owen \cite{low} provided the following construction of elation Laguerre planes from dual pseudo-ovals.
 Let $O$ be an $n$-dimensional dual pseudo-oval in $\PG(3n-1,q)$. Embed $\PG(3n-1,q)$ as a hyperplane $H$ in $\PG(3n,q)$, and let $\cL(O)$ be the incidence structure with points the $2n$-spaces of $\PG(3n,q)$ meeting $H$ in an element of $O$, generators the elements of $O$, circles the points of $\PG(3n,q)$ not on $H$, and taking as the incidence relation the natural one. Then, $\cL(O)$ is an elation Laguerre plane of order $q^n$.
 
By a result of Joswig \cite[Theorem 5.3]{jos}, the converse holds: a finite elation Laguerre plane of order $q^n$ with kernel containing $\GF(q)$ is isomorphic to $\cL(O)$, for some $n$-dimensional dual pseudo-oval $O$ in $\PG(3n-1,q)$.



By \cite[Theorem 3]{ste3}, an elation Laguerre plane $\cL$ can be coordinatized as follows: there is a matrix-valued mapping $D:\GF(q)^n \cup \{\infty\} \rightarrow \mM_{3n, n}(q)$ (where $\mM_{3n,n}(q)$ is the set of all $3n \times n$ matrices over $\GF(q)$) such that the point set of $\cL$ is $(\GF(q)^n \cup \{\infty\}) \times \GF(q)^n$ and every circle is of the form $K_c=\{(z, c\cdot D(z)) : z\in \GF(q)^n \cup \{\infty\}\}$, where $c \in \GF(q)^{3n}$. More precisely, $D(\infty)=(\rm{I} \ \rm{O} \ \rm{O} )^t$, where $\rm{I}$ and $\rm{O}$ denote the $n\times n$ identity and the zero matrices, respectively, and  the matrix $D(z)$, $z\in \GF(q)^n$, can be written as $(h(z) \ g(z) \ \rm{I})^t$ with suitable $n\times n$ matrices $h(z)$ and $g(z)$. After this coordinatization, the Laguerre plane $\cL$ is denoted by $\cL(g,h)$. In particular, $\{g(z) : z\in \GF(q)^n\}$ is a spread set defining a translation plane of order $q^n$, which is the dual of the derived projective plane $\bP\cL_P$, with $P=(\infty, 0)$, and the matrix $h$ describes a pencil of hyperovals in $\bP\cL_{P}$ \cite[p. 316]{ste4}.

Steinke in \cite{ste3} takes advantage of the above coordinatization to show the following:
\begin{theorem}\label{th_3}\cite[Theorem 4]{ste3} 
An elation Laguerre plane of order $q^n$ is equivalent to a pseudo-oval in $\PG(3n-1,q)$.
\end{theorem}
In particular, when $q$ is even, an elation Laguerre plane of order $q^n$ can be constructed from a pseudo-oval in $\PG(3n-1,q)$ in the following way \cite{ste3}. Let $\cO$ be a pseudo-oval in $\PG(3n-1,q)$, $q$ even, with nucleus $N$. The elements of $\cO$ can be labelled by $\GF(q)^n \cup \{\infty\}$: $\cO=\{X_z: z\in \GF(q)^n \cup \{\infty\}\}$. The projective space $\PG(3n-1,q)$ can be coordinatized in such a way that $X_{\infty}$ and $N$ are generated by the columns of the $3n\times n$-matrices $(\rm{I} \ \rm{O} \ \rm{O} )^T$ and $(\rm{O} \ \rm{I} \ \rm{O} )^t$, respectively. In this setting, the remaining elements $X_z$, with $z\in\GF(q)^n$, can be assumed to be generated by $(h(z) \ g(z) \ \rm{I})^t$, for uniquely defined $n\times n$ matrices $h(z)$ and $g(z)$.
Let $\cO'=(\cO\cup\{N\})\setminus \{X_{\infty}\}$. Then, $\cO'$ is a pseudo-oval in $\PG(3n-1,q)$ with nucleus $X_{\infty}$. Let  $\Sigma=\<N,X_0\>$, and  set $\Delta_z=\Sigma\cap\<X_{\infty},X_z\>$ for $z\in \GF(q)^n$. Then, by \cite{thas}, $S= \{\Delta_z : z\in \GF(q)^n\}\cup \{N\}$ is a spread of $\Sigma$, whose associated spread set is $\{g(z): z\in \GF(q)^n\}$. This implies that $\pi(\cO',X_{\infty})$ is the translation plane defined by the above spread set. 
Consider the matrix-valued mapping $D:\GF(q)^n \cup \{\infty\} \rightarrow \mM_{3n,n}(q)$ with $D(z)=(h(z) \ g(z) \ \rm{I})^t$, $z\in \GF(q)^n$, and $D(\infty)=(\rm{I} \ \rm{O} \ \rm{O} )^t$. It turns out that the mapping $D$ can be used as above to construct an elation Laguerre plane $\cL(g,h)$, whose derived projective plane $\bP\cL_{(\infty, 0)}$ at $(\infty, 0)$ is the dual of $\pi(\cO',X_{\infty})$.

Since there is a unique Laguerre plane of order 4, there is a unique elation Laguerre plane of order 4; the same holds for order 9. The only two existing Laguerre planes of order 8 are both elation Laguerre planes. These last two results are computer-dependent.  Since there is a computer-free proof of the uniqueness of a translation plane of order 8 \cite{lun}, in Section \ref{sec_2} we give a computer-free proof that there are exactly two elation Laguerre planes of order 8 (Corollary \ref{cor_5}). We also apply the classification of translation planes of order 16 \cite{dr,rei} and the classification of ovals in the duals of these planes \cite{prs} to classify by computer the elation Laguerre planes of order 16 in Corollary \ref{cor_6}, finding that there are exactly three of them. The latter classification has already been obtained by Steinke in \cite{ste4} using the same arguments as above, but with a different method of proof. Corollaries \ref{cor_n5} and \ref{cor_n6} provide results for $q=32$ and 64.

\subsection{Wild subspaces}

We now recall some results from the unjustly neglected paper of Wild \cite{wil}. It would be nice to state Wild's results in full generality, but that would take us into the realm of quasifields and coordinatization of translation planes, which is unfortunately too far from our present purpose. So we refer the reader to the original paper \cite{wil} for the general definitions and statements. We only require the Desarguesian case.

Let $\mathfrak F$ be the vector space of all functions $f:\GF(q)\rightarrow \GF(q)$, $q$ even, such that $f(0)=0$. It is well known that each element of $\mathfrak F$ can be written as a polynomial in one variable of degree at most $q-1$. 

Given an oval $\cO$ in $\PG(2,q)$, $q$ even, it is possible to choose projective coordinates such that $\cO$ contains the points $\{(1,0,0), (1, 1, 1), (0, 1, 0)\}$\footnote{\ To be formally correct, $(x,y,z)$ should be replaced by $\<(x,y,z)\>$ to denote the point defined by the vector $(x,y,z)$, but we guess that the only result of putting angle brackets everywhere in the paper would be to burden it with a uselessly heavy notation. For this reason, we warn the reader that we will freely use this simplified notation throughout this paper.}. This yields that it can be written as
\[
\cO = \cD(f)= \{(1,t,f(t)):t\in \GF(q)\} \cup \{(0,1,0)\},\ \hbox{with nucleus}\ (0,0,1),
\]
where $f\in\mathfrak F$ with $f(1) = 1$, and such that $f_s:x\mapsto (f(x+s)+f(s))/x$ is a permutation of $\GF(q)\setminus\{0\}$, for any $s\in\GF(q)$.
Permutations of $\GF(q)$ with these properties are called {\em o-polynomials} \cite[p.185]{hir}. It turns out that the degree of an o-polynomial is at most $q-2$. If  $f(1)=1$ is not required but the other conditions are, then $f$ is an {\em o-permutation over} $\GF(q)$. Conversely, for every  $f\in\mathfrak F$ such that $f(1) = 1$ and $f_s:x\mapsto (f(x+s)+f(s))/x$ is a permutation of $\GF(q)\setminus\{0\}$, for any $s\in\GF(q)$, the point-set $\cD(f)=\{(1,t,f(t)):t\in\GF(q)\}\cup \{(0,1,0)\}$ is an oval in $\PG(2,q)$ with nucleus $(0,0,1)$ \cite[Theorem 8.22]{hir}.
In this setting, the oval $\cD(x^{1/2})$ is a {\em conic}, while  $\cD(x^{2})$ is a {\em pointed conic} of $\PG(2,q)$. 

An $n$-dimensional {\em Wild subspace} over $\GF(q)$, $q$ even, is a $n$-dimensional subspace $W$ of $\mathfrak F$.

\begin{remark}\label{triang}{\em
We point out that our definition of a Wild subspace coincides with the definition of a linear triangular system of ovals given in \cite{wil}, when the left-quasifield is $\GF(q^n)$. In particular, since every o-polynomial is actually a polynomial in one variable of degree at most $q-2$, the property (iii) in the latter definition implies that a  Wild subspace contains exactly one o-polynomial.}
\end{remark}

The {\em kernel} of an $n$-dimensional Wild subspace $W$ over $\GF(q)$ is the largest subfield of $\GF(q^n)$ over which $W$ is a subspace. Let $f$ be an o-polynomial over $\GF(q^n)$ and $W$ be the set of all $\GF(q^n)$ scalar multiples of $f$. Clearly, $W$ is an $n$-dimensional Wild subspace over $\GF(q)$ with kernel $\GF(q^n)$, and  the ovals $\cD(f)$, $f\in W$, are images of each other under the homologies of $\PG(2,q^n)$ with center $(0,0,1)$ and axis the line joining $(1,0,0)$ and $(0,1,0)$.

Since the Desarguesian projective plane  $\PG(2,q)$ is self-dual, as an immediate consequence of  Theorem 1 and Corollary 2 in \cite{wil}, we get the following results.

\begin{theorem}\label{th_1} Let $\cO$ be an $n$-dimensional pseudo-oval in $\PG(3n-1,q)$, $q$ even, and $X$ be an element of $\cO$. Suppose that the translation plane $\pi(\cO,X)$ is Desarguesian. Then, there arises an $n$-dimensional Wild subspace $W(\cO,X)$ over $\GF(q)$. Conversely, each $n$-dimensional Wild subspace $W$ over $\GF(q)$ gives rise to an $n$-dimensional pseudo-oval $\cO$ of $\PG(3n-1,q)$ such that $\pi(\cO,X)$ is Desarguesian, for some element $X$ of $\cO$, with $W=W(\cO,X)$.
\end{theorem}

\begin{theorem}\label{th_2} Let $\cO$ be an $n$-dimensional pseudo-oval in $\PG(3n-1,q)$, $q$ even, and $X$ be an element of $\cO$ such that the translation plane $\pi(\cO,X)$ is Desarguesian. Then, $W(\cO)$ has kernel $\GF(q^n)$ if and only if $\cO$ has kernel $\GF(q^n)$ (i.e., $\cO$ is elementary).
\end{theorem}

\section{Machinery and computational results}\label{sec_2}

In this section, we classify all pseudo-ovals in $\PG(3n-1,2)$, $n=3,4$, and all pseudo-ovals in $\PG(3n-1,2)$, $n=5,6$, that give a Desarguesian translation plane. In the latter two cases, we require that there is a Desarguesian translation plane since the classification of translation planes and that of the ovals in their duals are not complete. As a by-product, thanks to the correspondence between pseudo-ovals and TGQs, we also classify all TGQs of order $q^n$, $n=3,4$, and all TGQs $T(\cO)$ of order $q^n$, $n=5,6$, when $\cO$ gives a Desarguesian translation plane $\pi(\cO,X)$, for some $X\in\cO$. Another consequence, via Theorem \ref{th_3}, is the classification of all elation Laguerre planes of order $2^n$, $n=3,4$, and that of all elation Laguerre planes of order $2^n$, $n=5,6$, with a Desarguesian projective completion.
In light of Theorems \ref{th_1} and \ref{th_2}, to achieve the above aim, we actually classify the $n$-dimensional Wild subspaces over $\GF(q)$ for $n=3,4,5,6$.

Our programs, written in MAGMA \cite{magma}, were run on a macOS system with one quad-core Intel Core i7 2.2 GHz processor. All the computational operations take approximately 60 hours of CPU time. The code is available on email request to the 3rd author.

In \cite[Lemma 1]{okp3} an action of $\PGammaL(2,q)$, $q$ even, on the vector space  $\mathfrak F$ is defined: it is  known as {\em magic action}. Recall that, if $f(x)=\sum{a_ix^i}$ and $\gamma\in\Aut(\GF(q))$, then  $f^\gamma(x)=\sum{a_i^\gamma x^i}$.

Any element $\psi=(A,\gamma)\in\PGammaL(2,q)$, with $A=\begin{pmatrix}a& b \\ c & d\end{pmatrix}$ and $\gamma\in\Aut(\GF(q))$, acts on $\mathfrak F$ by mapping $f$ to $\psi f$, where 
\[
\psi f(x)=|A|^{-1/2}\left[(bx+d)f^\gamma\left(\frac{ax+c}{bx+d}\right)+b\,x\,f^\gamma\left(\frac{a}{b}\right)+d\,f^\gamma\left(\frac{c}{d}\right)\right];\]
here, if $b=0$ then $bf^\gamma(a/b) = 0$ by convention (and similarly for $df^\gamma(c/d)$).

This action of $\PGammaL(2,q)$ on $\mathfrak F$ is called  {\em magic action}. 

\begin{theorem}\cite[Theorem 4,Theorem 6]{okp3}\label{th_14}
The magic action preserves the set of o-permutations, and for $g=\psi f$ we have that  $\cD(f)$ and $\cD(g)$ are projectively equivalent under a collineation $\overline\psi$ of $\PG(2,q)$ induced by $\psi$. Conversely, if $\cD(f)$ and $\cD(g)$, with $f,g$ o-permutations, are equivalent under $\PGammaL(3,q)$, then there is $\psi\in\PGammaL(2,q)$ such that $\psi f\in\<g\>$.
\end{theorem}

For an o-polynomial, there are $q - 1$ o-permutations, namely the non-zero scalar multiples of the o-polynomial, and for an o-permutation $f$ there is the unique o-polynomial $(1/f(1))f$. For $f \in \mathfrak F$ , $\<f\>$ will denote the one-dimensional subspace of $\mathfrak F$ containing $f$.

Since the magic action preserves the set of o-permutations, we say that two o-permutations $f$ and $g$ are {\em equivalent} if they correspond under the magic action, i.e., there exists $\psi\in \PGammaL(2,q)$ such that $\psi f\in\<g\>=\{\lambda g:\lambda\in\GF(q)\}$. 

The main theorem of Segre \cite{se} is that, if $q$ is odd, then every oval of $\PG(2, q)$ is a conic. It is also shown in \cite{se2} that, if $q$ is 2 or 4, then every oval of $\PG(2,q)$ is a conic, and if $q=8$, then every oval of $\PG(2,q)$ is either a conic or a pointed conic. Other ovals that we need to mention are the Lunelli-Sce ovals in $\PG(2,16)$ \cite{lusce}. They are uniquely specified by being ovals of $\PG(2,16)$ that are neither conics nor pointed conics. This is due to the classification of ovals in $\PG(2,16)$ obtained by Hall \cite{ha} with the use of a computer, and by O'Keefe and Penttila \cite{okp} without it.
The ovals of $\PG(2,32)$ were classified: up to projective equivalence, there are 35 ovals, coming from 6 hyperovals \cite{pr}. The hyperovals of $\PG(2,64)$ have been classified by Vandendriessche in \cite{vander}: up to projective equivalence, there are 4 hyperovals, giving rise to 19 ovals \cite{pent}.

Surveys of the known families of hyperovals in $\PG(2,q)$ appear in \cite{che, cokp}. They are the translation hyperovals \cite{se2}, the Segre-Bartocci hyperovals \cite{se3, sb} for $q$ not square, the two families of Glynn hyperovals \cite{gly} for $q$ not square, the Payne hyperovals \cite{p3} for $q$ not square, the Subiaco hyperovals \cite{cppr} for all even $q$ (two families \cite{ppp} if $q$ is a square but not a fourth power), the Cherowitzo hyperovals \cite{che} for $q$ not square (proved to be new in \cite{pp}), the Adelaide hyperovals \cite{cokp} for $q$ square, and the O'Keefe-Penttila hyperoval \cite{okp2} in $\PG(2, 32)$. It is probably also worth passing the remark that the Lunelli-Sce ovals are no longer sporadic, having been generalized to the Subiaco ovals in \cite{cppr} and to the Adelaide ovals in \cite{cokp}. No other hyperovals of $\PG(2, q)$ are known.

The ovals of the four projective planes of order 9 were classified in \cite{den} (see also \cite{ckk}). The hyperovals of the twenty-two known projective planes of order 16 (including all translation planes of order 16 \cite{che2,gs} and all dual translation planes of order 16) were classified in \cite{prs}. It is worth remarking that four of the known planes of order 16 contain no ovals.
For no other projective plane of finite order have the ovals in that plane been classified.

In what follows, we consider the projective plane $\PG(2,2^n)$,  $n=3,4,5,6$. Since every oval lies in a unique hyperoval, all the ovals in $\PG(2,2^n)$, up to equivalence, can be obtained by removing one representative from each point-orbit of the stabilizer in $\PGammaL(3,2^n)$ of each hyperoval, viewed as a permutation group acting on it. We apply a suitable linear collineation to every oval in such a way the corresponding image contains the points $(0, 1,0), (1,0,0), (1, 1, 1)$ and has nucleus $(0, 0, 1)$. Then, we obtain the associated o-polynomials by Lagrange interpolation. Via the magic action, we get all the ovals in $\PG(2,2^n)$ on the fundamental quadrangle, and we store the corresponding o-polynomials in terms of their coefficients. This operation, using MAGMA, produces all isomorphism classes of ovals in $\PG(2,2^n)$, whose number is denoted by $k_n$; see \cite{se2} for $n=3$, \cite{ha, okp} for $n=4$, \cite{pr} for $n=5$, and \cite{pent} for $n=6$. 

\begin{lemma}
If $W$ is a Wild subspace, and $f,g \in W$ with $f \ne g$, then $f(1)\ne g(1)$.
\end{lemma}

\begin{proof}
From the definition of a Wild subspace, it follows that $f +g \in W$, so $f+g$ is an o-permutation (or $0$). If $f \ne g$, then $f+g \ne 0$, so $(f+g)(1) \ne 0$, i.e.,  $f(1) \ne g(1)$.
\end{proof}
\begin{corollary}\label{cor_1}
If $W$ is an $n$-dimensional Wild subspace over $\GF(q)$, then  $\{f(1):f \in W\}=\GF(q^n)$.
\end{corollary}
\begin{proposition}
The only $n$-dimensional Wild subspace over $\GF(2)$, $n=3,4,5,6$, is $\GF(2^n)$, whose kernel is clearly $\GF(2^n)$.
\end{proposition}
\begin{proof}
By Corollary \ref{cor_1}, we can fix an element $a$ of $\GF(2^n)$ with $a \ne 0,1$.
From the definition of Wild subspaces and Remark \ref{triang}, under the magic action, we may assume that a Wild subspace $W$ contains a fixed o-polynomial $f$ from some of the $k_n$ equivalence classes of ovals. Then, there exists $g \in W$, $g\ne f$, with $g(1)=a$. Thus, $h=a^{-1}g$ is an o-polynomial, and $(1+a)^{-1}(f+g)$ is an o-polynomial as well.
Therefore, we may say that, for the fixed o-polynomial $f$ and the element $a$, there is an o-polynomial $h$ such that $p=(1+a)^{-1}(f+ah)$ is an o-polynomial. This implies that $p_s:x\mapsto (p(x+s)+p(s))/x$ is a permutation of $\GF(2^n)\setminus\{0\}$, for any $s\in\GF(2^n)$. By using MAGMA, we checked that for $n=3,4,5,6$, and any fixed o-polynomial $f$ from each of the $k_n$ classes, the only o-polynomial $h$ that passes the test on $p_s$ being a permutation as above is $f$ itself.
\end{proof}  
The following result is an immediate consequence of Theorems \ref{th_1} and \ref{th_2}.
\begin{corollary}
Every pseudo-oval $\cO$ of $\PG(3n-1,2)$, $n=3,4,5,6$, such that $\pi(\cO,X)$ is a Desarguesian translation plane, for some element $X\in \cO$, is elementary.
\end{corollary}

In the following subsections, we give all computational details for each value $n$ that has been taken into consideration.

\subsection{Case $n=3$}

\begin{theorem} \label{th_4}
Up to projective equivalence, there are exactly two $3$-dimensional pseudo-ovals in $\PG(8, 2)$, and each of them is elementary. Explicitly, either $\cO$ arises from a conic $\cC$ or from a pointed conic $\mathcal{PC}$.
\end{theorem}
\begin{proof} Let $\cO$ be a $3$-dimensional pseudo-oval in $\PG(8, 2)$ and $X$ be an element of $\cO$. Then, $\pi(\cO,X)$ is a translation plane of order $8$. There is a computer-free proof that all translation planes of order $8$ are Desarguesian \cite{lun}. Thus, the hypotheses of Theorem \ref{th_1} are
satisfied, and we need only to classify the $3$-dimensional Wild subspaces over $\GF(2)$. The $70$ o-permutations over $\GF(8)$ are given, for example, in \cite[Theorem 2.2]{okp4}. It follows by inspection with MAGMA that every $3$-dimensional Wild subspace over $\GF(2)$ has kernel $\GF(8)$. Thus, by Theorem \ref{th_2}, $\cO$ has kernel $\GF(8)$, and so $\cO$ arises from an oval in $\PG(2, 8)$, which is either a conic or a pointed conic by \cite[Theorem 2.2]{okp4},\cite{se2}. Since these two ovals are not projectively equivalent, the result follows from Proposition \ref{prop_1}.
\end{proof}

\begin{corollary} \label{cor_2}
Up to isomorphism, there are exactly two TGQs of order 8, namely $Q(4, 8) = T_2(\cC)$ and $T_2(\mathcal{PC})$.
\end{corollary}

\begin{corollary}\label{cor_5}
Up to isomorphism, there are exactly two elation Laguerre planes of order 8, namely the Miquelian plane $\cL(\cC)$ and the ovoidal plane $\cL(\mathcal{PC})$.
\end{corollary}
\begin{remark} {\em There are exactly two Laguerre planes of order 8. However, this result is computer-dependent. There seems to be some value in having a computer-free proof of the weaker result here.}
\end{remark}

\subsection{Case $n=4$}

\begin{theorem} \label{th_5}
Up to projective equivalence, there are exactly three 4-dimensional pseudo-ovals in $\PG(11, 2)$, and each of them is elementary. Explicitly, $\cO$ arises from a
conic $\cC$, a pointed conic $\mathcal{PC}$, or a Lunelli-Sce oval $\mathcal{K}$.
\end{theorem}
\begin{proof} According to the coordinatization given in \cite[Theorem 3]{ste3}, the derived projective plane $\bP\cL_P$, with $P=(\infty, 0)$, which is a dual translation plane, contains a pencil of hyperovals through the points $\omega$, $(0)$, $(0,0)$, where $\omega$ is the translation center of $\bP\cL_P$. \cite[p. 316]{ste4}. The translation planes of order 16 were classified in \cite{dr, rei}. There are eight distinct translation planes of order 16.  The hyperovals in these planes and in their duals were classified in \cite{che2, prs}. The only (dual) translation plane admitting a pencil of hyperovals with the above property is $\PG(2,16)$. Hence, each derived projective plane of an elation Laguerre plane of order $16$ is Desarguesian \cite[Proposition 5]{ste4}. Under the aforementioned correspondence between elation Laguerre planes and pseudo-ovals (see Theorem \ref{th_3} for $q$ even), we know that for all pseudo-ovals  $\cO$ in $\PG(11,2)$ there is $X\in \cO$ such that $\pi(\cO,X)$ is Desarguesian.
Thus, the hypotheses of Theorem \ref{th_1} are satisfied, and we need to classify all $4$-dimensional Wild subspaces $W(\cO)$ over $\GF(2)$. 
The $30870$ o-permutations over $\GF(16)$, split into $k_4=3$  (see  \cite{ha,okp}) isomorphism classes of ovals in $\PG(2,16)$, are obtained via the magic action. It follows by inspection with MAGMA that every $4$-dimensional Wild subspace over $\GF(2)$ has kernel $\GF(16)$. Thus, by Theorem \ref{th_2}, $\cO$ has kernel $\GF(16)$, and so $\cO$ arises from an oval in $\PG(2, 16)$, which is or a conic or a pointed conic or a Lunelli-Sce oval  \cite{ha,okp}.
Since these three ovals are not projectively equivalent, the result follows from Proposition \ref{prop_1}.
\end{proof}

\begin{corollary}\label{cor_4} 
Up to projective equivalence, there are exactly three $2$-dimensional pseudo-ovals
in $\PG(5, 4)$, and each of them is elementary. Explicitly, $\cO$ arises from a conic
$\cC$, a pointed conic $\mathcal{PC}$ or a Lunelli-Sce oval $\mathcal{K}$.
\end{corollary}

\begin{corollary}\cite[Theorem 8]{ste4}\label{cor_6} Up to isomorphism, there are exactly three elation Laguerre planes of order
$16$, and these are ovoidal. Explicitly, they are the Miquelian plane $\cL(\cC)$, the ovoidal plane
$\cL(\mathcal{PC})$ over a pointed conic, and the ovoidal plane $\cL(\mathcal{K})$ over a Lunelli-Sce oval.
\end{corollary}

\begin{remark}{\em In \cite{gs} the authors classify translation Laguerre planes of order $16$ by computer, finding just $\cL(\cC)$ and $\cL(\mathcal{PC})$. Of course, $\cL_2(\cO)$ is a translation Laguerre plane if and only if $\cO$ is a translation oval.}
\end{remark}

\begin{corollary} \label{cor_3}
Up to isomorphism, there are exactly three TGQs
of order $16$, and these are of the form $T_2(\cO)$, where $\cO$ is an oval in $\PG(2, 16)$. Explicitly,
the three TGQs of order $16$ are $Q(4, 16) = T_2(\cC)$,
$T_2(\mathcal{PC})$, and $T_2(\mathcal{K})$.
\end{corollary}

In both following cases, it is worth underlining that, since the classification of translation planes of order $2^4$ and $2^5$, and that of ovals in their duals are not complete, we cannot adopt the same approach as in case $n=4$. Therefore, we limit ourselves to the classification of all pseudo-ovals $\cO$ of $\PG(3n-1,2)$, $n=5,6$, which admit a Desarguesian translation plane. Consequently, we classify all TGQs $T(\cO)$ of order $q^n$, $n=5,6$, when $\cO$ admits a Desarguesian translation plane, and all elation Laguerre planes of order $2^n$, $n=5,6$, with a Desarguesian projective completion.

\subsection{Case $n=5$}

\begin{theorem}\label{th_n5} 
Up to projective equivalence, there are exactly thirty-five $5$-dimensional pseudo-ovals in $\PG(14, 2)$ admitting a Desarguesian translation plane, and each of them is elementary.
\end{theorem}
\begin{proof}
Let $\cO$ be a $5$-dimensional pseudo-oval in $\PG(14, 2)$  such that the translation plane $\pi(\cO,X)$ is Desarguesian for some $X\in \cO$. Thus, the hypotheses of Theorem \ref{th_1} are satisfied, and we need to classify all $5$-dimensional Wild subspaces $W(\cO)$ over $\GF(2)$. Via magic action, we obtain $3537700$ o-permutations over $\GF(32)$, split into $k_5=35$ isomorphism classes of ovals in $\PG(2,32)$ \cite{pr}. The above-described inspection with MAGMA gives that every $5$-dimensional Wild subspace over $\GF(2)$ has kernel $\GF(32)$. Thus, by Theorem \ref{th_2}, $\cO$ has kernel $\GF(32)$, and so $\cO$ arises from an oval in $\PG(2, 32)$. Since in $\PG(2,32)$ there are exactly thirty-five ovals which are pairwise not projectively equivalent, the result follows from Proposition \ref{prop_1}.
\end{proof}
\begin{corollary}\label{cor_n5} 
Up to isomorphism, there are exactly thirty-five elation Laguerre planes of order
$32$ with Desarguesian projective completion, and these are ovoidal. 
\end{corollary}

\begin{corollary}\label{n_5}
Up to isomorphism, there are exactly thirty-five TGQs $T(\cO)$ of order $32$, with $\cO$ admitting a Desarguesian translation plane, and these are of the form $T_2(\cO)$, where $\cO$ is an oval in $\PG(2, 32)$.
\end{corollary}

\subsection{Case $n=6$}

\begin{theorem}\label{th_n6} 
Up to projective equivalence, there are exactly nineteen $6$-dimensional pseudo-ovals in $\PG(17, 2)$ admitting a Desarguesian translation plane, and each of them is elementary.
\end{theorem}
\begin{proof}
Let $\cO$ be a $6$-dimensional pseudo-oval in $\PG(17, 2)$  such that the translation plane $\pi(\cO,X)$ is Desarguesian for some $X\in \cO$. Thus, the hypotheses of Theorem \ref{th_1} are satisfied, and we need to classify all $6$-dimensional Wild subspaces $W(\cO)$ over $\GF(2)$. Via magic action, we obtain $17297346$ o-permutations over $\GF(64)$, split into $k_6=19$ isomorphism classes of ovals in $\PG(2,64)$ \cite{pent}. The above-described inspection with MAGMA gives that every $6$-dimensional Wild subspace over $\GF(2)$ has kernel $\GF(64)$. Thus, by Theorem \ref{th_2}, $\cO$ has kernel $\GF(64)$, and so $\cO$ arises from an oval in $\PG(2, 32)$. Since in $\PG(2,64)$ there are exactly nineteen ovals which are pairwise not projectively equivalent, the result follows from Proposition \ref{prop_1}.
\end{proof}
\begin{corollary}\label{cor_n6} 
Up to isomorphism, there are exactly nineteen elation Laguerre planes of order
$64$ with Desarguesian projective completion, and these are ovoidal. 
\end{corollary}

\begin{corollary}\label{n_6} 
Up to isomorphism, there are exactly thirty-five TGQs
of order $64$, and these are of the form $T_2(\cO)$, where $\cO$ is an oval in $\PG(2, 64)$.
\end{corollary}

\end{document}